\documentclass[11pt]{article}
\usepackage{amsmath,amssymb,amsthm}
\usepackage{graphicx}
\usepackage{array}   
\newcolumntype{C}{>{$}c<{$}} % math-mode version of "c" column type

\usepackage{xcolor}

\theoremstyle{plain}
\newtheorem{theorem}{Theorem}
\newtheorem*{theorem*}{Theorem}
\newtheorem{lemma}{Lemma}
\newtheorem{corollary}{Corollary}
\newtheorem*{corollary*}{Corollary}
\newtheorem{definition}{Definition}

\newtheorem*{conjecture*}{Conjecture}

\theoremstyle{definition}
\newtheorem{example}{Example}
\newtheorem*{remark*}{Remark}

\def\beq{\begin{equation}}
\def\eeq{\end{equation}}
\def\ds{\displaystyle}

\def\-{\backslash}

\def\D{\mathcal{D}}

\def\L{\Lambda}

\def\a{\alpha}

\def\nx{\overline{x}}
\def\nt{\overline{t}}
\def\pfi{\varphi}

\title{Functional completeness and primitive positive decomposition of relations on finite domains}
\author{Sergiy Koshkin\\
\\
Department of Mathematics and Statistics\\
University of Houston-Downtown\\
One Main Street\\
Houston, TX 77002\\
e-mail: koshkins@uhd.edu}
\date{}
\begin{document}

\maketitle
\begin{abstract} 

We give a new and elementary construction of primitive positive decomposition of higher arity relations into binary relations on finite domains. Such decompositions come up in applications to constraint satisfaction problems, clone theory and relational databases. The construction exploits functional completeness of 2-input functions in many-valued logic by interpreting relations as graphs of partially defined multivalued `functions'. The `functions' are then composed from ordinary functions in the usual sense. The construction is computationally effective and relies on well-developed methods of functional decomposition, but reduces relations only to ternary relations. An additional construction then decomposes ternary into binary relations, also effectively, by converting certain disjunctions into existential quantifications. The result gives a uniform proof of Peirce's reduction thesis on finite domains, and shows that the graph of any Sheffer function composes all relations there. 
 
\bigskip

\textbf{Keywords}: relational operations, primitive positive formula, coclone, constraint satisfaction problem, irreducible relation, many-valued logic, functional completeness, Sheffer functions, Post algebra, relative product, hypostatic abstraction, Peirce's reduction thesis
\medskip

\textbf{MSC}: 08A02, 03G20, 08A70, 03B50, 08A40
\end{abstract}

\section*{Introduction}

First examples of primitive positive definitions appeared at the dawn of mathematical logic in de Morgan's compositions (relative products) of binary relations, and were originally motivated by linguistic expressions like ``brother of a parent". The problem of decomposing higher arity relations into binary relations was studied soon after de Morgan by C.S. Peirce \cite{Brun91}. In a more formal context, primitive positive formulas (ppfs), those that apply only conjunctions and existential quantifiers to predicates, appeared in Robinson's work on model theory in the 1950-s \cite{Rob62}. They came up independently in universal algebra in the 1960-s in the context of Pol-Inv Galois connection \cite{Born08,Lau}, as preserving invariance of relations under a function. 

Since the 1970-s, ppfs appeared in more practical contexts, as conjunctive queries in relational databases \cite{Kosh23} and as templates for polynomial time reductions of constraint satisfaction problems (CSP) \cite{CreVol08,JCG97,Sch78}. Schaefer exploited the latter in his proof of the dichotomy theorem for Boolean CSP \cite{Sch78}, and ppfs are featured prominently in Jeavons's algebraic approach to CSP since the 1990-s \cite{BJK,JCG97}. In particular, a recent generalization of Schaefer's CSP dichotomy to arbitrary finite domains relies on primitive positive reductions \cite{B3KZ}.

We will be mainly interested in the perspective from Peirce's classical work, which was motivated, in part, by a functional analogy. It is a trivial observation that composition of $1$-input functions does not produce functions with more inputs, while composition of $2$-input functions can produce $n$-input functions with any $n\geq2$. Let us call the functional reduction thesis (FRT) the two-part claim that $n$-input functions decompose into (reduce to) $2$-input functions, but the latter do not all decompose into $1$-input ones (are irreducible). While the irreducibility part is trivial, proving the reducibility part needs some work. 

On Boolean domains, it follows from the existence of disjunctive normal forms, with only $\land$, $\lor$ and $\neg$ composed, which are implicit in Peirce's work of 1880-s and were explicitly derived by Schr\"oder in 1890, see \cite{Lew18}. 
The case of arbitrary finite domains was settled in Post's seminal 1920 dissertation, where the general question about functionally complete sets of functions was also raised, and solved in the Boolean case \cite{Urq09}. Post proved FRT by introducing a special class of many-valued logics with cyclic negation, and generalizing truth tables and disjunctive normal forms to them. The case of infinite domains was settled only in 1945 by Sierpi\'nsky \cite{Sier45}, but, surprisingly, it turned out to be simpler.

Even before these results, Peirce suggested that such reductions transfer from functions to relations. Interpreting, say, $f(t,z)$ and $g(x,y)$ as relations, $u=f(t,z)$ and $t=g(x,y)$, respectively, we can express their composition $f(g(x,y),z)$ as
\begin{equation}\label{GraphProd}
\big[u=f(g(x,y),z)\big]=\exists t\big[\big(u=f(t,z)\big)\land\big(t=g(x,y)\big)\big],
\end{equation}
a special case of the relative product of arbitrary relations, e.g.
\begin{equation}\label{RelProd}
R(x,y,z,u)=\exists t\big[S(x,y,t)\land T(t,z,u) \big].
\end{equation}
Compositions of more than two functions are generalized by expressions like the ``triple junction" \cite{Herz}:
\begin{equation}\label{3junc}
R(x,y)=\exists t\big[F(t,x)\land G(t,y)\land H(t,z)\big].
\end{equation}
Using identity relations, one can also express conjunctions with identified variables, like $I_3(x,y,z):=[(x=y)\land(y=z)]$, i.e. joins of the relational database theory. The primitive positive composition of relations, pp-composition for short, is then seen as a natural extension of functional composition.

By analogy, let us call the relational reduction thesis (RRT) the two-part claim that $n$-ary relations for $n\geq2$ reduce to binary relations, but not all binary relations reduce to unary ones. This is equivalent to a part of Peirce's reduction thesis formulated even before the more elementary FRT \cite{Kosh22}. Peirce gave a clever general construction of reductions, which he called ``hypostatic abstraction", that, suitably modified, works on infinite domains. However, on finite domains, it requires adjoining additional elements and extending original relations to the larger domain, see Section \ref{Inf}.

One might expect that RRT on finite domains could be analogously derived from standard results on relational completeness, but such results are largely missing. For example, there are no general analogs of either Post's or Slupecki's criteria \cite{Ros77} for relationally complete systems. Nonetheless, RRT on finite domains has been resolved by other means, but the results are scattered in the literature, often formulated in different terms and proved by non-elementary methods.

The most non-trivial reducibility part follows from an auxiliary construction of a pp-complete system of two binary relations in the seminal 1969 work of BKKR (at the end of part II) that introduced the Pol-Inv Galois connection \cite{BKKR}. One can also derive the result non-constructively from the Galois connection itself \cite[1.1.22]{PosKal}. That RRT {\it fails} on Boolean domains follows from another landmark, Schaefer's 1978 paper that proved the dichotomy theorem for Boolean CSP \cite{Sch78}. Namely, Schaefer proved in passing that all binary relations on Boolean domains are bijunctive while some ternary relations are not, and bijunctivity is preserved by the pp-composition. However, none of the above or subsequent works spell out the result (\cite{HerPos04} comes closest) or connect RRT to FRT.

Is there a connection? In a 1906 letter to James, Peirce declared that they are ``one and the same". Even making room for rhetorical flourishes, this is much too rash. Not all relations are functions, and  pp-composition of graphs from graphs does not always come from composition of their functions. On infinite domains, hypostatic abstraction bypasses FRT altogether rather than builds on it, so it is unclear why FRT even suggests RRT. Note that the connection that is supposed to be at play here is an elementary one -- treating graphs of {\it individual} functions {\it as} relations, see \eqref{GraphProd}. It is quite distinct from the Galois connection between clones and coclones, which has a {\it class} of functions {\it preserving} a class of relations.

In this paper, we will establish the desired connection and give elementary proofs for all cases of RRT based on it. Unlike the BKKR construction, it provides a wide selection of pp-complete systems, rather than a single one, and capitalizes on well-developed methods of functional decomposition. In fact, graphs of functions in any functionally complete system form a relationally pp-complete system (Theorem \ref{FunRed}). This is somewhat surprising because there are many non-graph relations, and they cannot always be pp-composed from graphs in general coclones. As a simple consequence, we obtain the existence of\, `Sheffer relations' that alone form a pp-complete system. The graph of the Sheffer stroke is such a relation on Boolean domains. The proof is based on treating relations as multivalued partially defined `functions' that are then pp-decomposed into ordinary 2-input functions, e.g. by using canonical normal forms in Post's many-valued logics. 

Conceptually, our construction  explains why pp-reductions to ternary relations are always possible (assuming that reducibility of functions to $2$-input functions is taken for granted). From the practical point of view, it is of interest because pp-composition is less intuitive than functional composition and computational methods for it are much less developed. Moreover, constructing convenient sets of relations that are pp-complete in coclones (their {\it bases}) has received much attention in the recent CSP literature \cite{BRSV05,Lag17}.

However, passing from functions to their graphs increases arity by $1$, so existence of functionally complete systems of $2$-input functions translates into reducibility of higher arity relations to {\it ternary}, not binary, relations. This is as it should be in general, given that pp-irreducible ternary relations do exist on Boolean domains ($\neg I_3$ is one). As a supplement, we give another elementary construction that reduces cofinite relations of small arity to binary relations (Theorem \ref{CoBin}). In a sense, it is dual to hypostatic abstraction, and converts disjunctions of unary cosingletons into existential quantifications over certain binary predicates on non-Boolean domains. For completeness, we also give an elementary discussion of the Boolean case and relate pp-composition to Peirce's original ``bonding" that deviates from it on ternary relations.

The paper is organized as follows. In Section \ref{Inf} we recall hypostatic abstraction that reduces higher arity relations to binary ones after extending the domain, and L\"owenheim-Sierpi\'nsky's pairing construction that does the same for functions, proving RRT and FRT on infinite domains. In Section \ref{FunRel} we interpret relations with a distinguished argument (called {\it relatives}) as multivalued partially defined `functions' and reduce higher arity relations to ternary ones on arbitrary domains by exploiting FRT. Section \ref{PostLog} introduces Post's logics and illustrates the algorithm for converting functional into relational reductions by our method. In Section \ref{RedBin} we give the `existentialization of disjunctions' construction that reduces cofinite relations of small arity to binary ones on domains with at least three elements. In Section \ref{BoolIrr} we discuss bijunctive relations and give an elementary proof of pp-irreducibility of some ternary relations on Boolean domains. Section \ref{TerId} introduces Peirce's notion of bonding and his conversion of pp-reductions to bond reductions for arities $4$ and higher by clever use of ternary identities. In the final section, we summarize our conclusions and state some open problems.

\section{Preliminaries}\label{Prelim}

We use the standard notation and terminology of set theory and predicate calculus. Relations are defined on a set $\D$ called the {\bf domain} and are subsets of its Cartesian powers $\D^n:=\D\times\dots\times\D$. The cardinality of a set $S$ is denoted $|S|$. As is well known, $|\D^n|=|\D|^n$. When $R\subseteq\D^n$ the number $n$ is called the relation's {\bf arity} and the relation is called $n$-ary relation. For $n=1,2,3,4$ we use the shorthands unary, binary, ternary, quaternary, respectively. Relations are called {\bf finite} when they have finite cardinality, {\bf small} when their cardinality is no more than that of the domain, {\bf cofinite} when their complements are finite, and {\bf finitary} when they have finite arity. We only consider finitary relations.

Elements of $\D^n$ are called $n$-tuples or just {\bf tuples}, when $n$ is understood or immaterial. Small Latin letters are typically used for elements of $\D$, and small Greek letters for tuples in $\D^n$. If $\a\in\D^n$ it's $i$-th member is denoted $\a_i$. We adopt the usual convention of canonically identifying tuples of tuples with longer tuples, and hence of identifying $\D^n\times\D^{m}$ with $\D^{n+m}$, and so on. Some standard $n$-ary relations are the {\bf universal relations} $U_n:=\D^n$ with all $n$-tuples, and the {\bf identity relations} $I_n$ with all and only $n$-tuples of identical elements. Occasionally, we abbreviate the latter as $\vec{a}:=(a,\dots,a)$. 

Unless otherwise stated, all predicates will be interpreted on a domain, and relations will be identified with the predicates they interpret. In particular, the same letter will be used for a relation and its predicate, i.e. $R(\a)=R(\a_1,\dots,\a_n)$ will mean the same as $\a\in R$, and $\neg R$ will denote the complement of $R$ in $\D^n$. Furthermore, we identify relations with  their truth value functions, $0$\,-$1$ functions on $\D^n$. Logical connectives, conjunction $\land$, disjunction $\lor$, etc., will be used with the usual meaning except in Section \ref{PostLog}, where they are explicitly redefined in many-valued logics. 

Technically, {\bf primitive positive formulas (ppfs)} are those with only existential quantifiers in the prefix and only conjunctions as connectives in the quantifier-free part \cite{CreVol08,Lag17,Rob62}. However, any formula with only existential quantifiers and conjunctions can be converted into such a form by prenexing quantifiers, so we will apply the term more broadly to all such formulas. In particular, our ppfs are closed under taking conjunctions and existentially quantifying. The following definition introduces the central concepts used in the paper.
\begin{definition}[{\bf pp-composition}]\label{ClonComp}
A relation is a {\bf composition of relations}, or {\bf decomposes} into them, when its predicate can be expressed by a ppf that contains only the predicates of the composing relations. A decomposition is a {\bf reduction} when the composing relations have strictly lower arity than the composite. A relation is {\bf reducible} when it admits a reduction. 
\end{definition}
\noindent We prefer ``pp-composition" to the more common ``pp-definition" to emphasize the functional analogy and algebraic aspects, and drop ``pp" when context allows. Cartesian products (free conjunctions), joins and projections (existential quantifications) of the relational database theory, relative products \eqref{RelProd}, triple \eqref{3junc} and higher junctions are all special cases of pp-composition. {\bf Clones} are sets of functions containing projections and closed under composition, {\bf coclones} are sets of relations containing diagonals and closed under pp-composition, they are dual to each other under the Pol-Inv Galois connection \cite{Lau,PosKal}.

\section{The infinite we do immediately...}\label{Inf}

In this section we will describe Peirce's hypostatic abstraction that reduces $n$-ary relations to binary ones, and the pairing construction of L\"owenheim and Sierpi\'nsky that does the same for functions. Both require the domain to be infinite to work without extending it. This is one of those cases that prompted Erd\"os and Tarski to quip \cite{Zus16}:``The infinite we do immediately, the finite takes a little time."

Given a domain $\D$ and a relation $R(x_1,\dots, x_n)$ on it, extend the domain by $n$-tuples of elements of $\D$, and define binary relations $R_i(t,x)$ to hold if and only if there is a tuple $t=(x_1,\dots, x_n)\in R$ with $x=x_i$.  `Abstract' tuples are hypostatized into new elements here, hence the name. By construction, 
\begin{equation}\label{HypAbs}
R(x_1,\dots,x_n)=\exists t\big[R_1(t,x_1)\land\dots\land R_n(t,x_n)\big],
\end{equation}
and an arbitrary $n$-ary relation is reduced to binary ones by this ppf. The domain extension can sometimes be avoided using the following trick due to Herzberger \cite{Herz} (Sierpi\'nsky used a similar trick for functions much earlier). Instead of adjoining the tuples, define an injective map $\tau:R\to\D$ that uniquely `codes' them into elements of $\D$, and redefine $R_i(t,x)$ into a relation on the original $\D$ that holds when $t=\tau(x_1,\dots,x_n)$ is in the range of $\tau$ and $x=x_i$. Then \eqref{HypAbs} still holds and no domain extension is needed.

For such encoding map to exist, we clearly need $|R|\leq|\D|$, such relations were termed {\it small} in \cite{Kosh22}. In general, $|R|\leq|\D|^n$ for $n$-ary $R$, and when $\D$ is infinite $|\D|^n=|\D|$ for any $n$. In other words, any (finitary) relation on an infinite domain is small and, hence, reducible.  However, in general, the requisite bijections between $\D$ and $\D^n$ are non-constructive. By a theorem of Tarski \cite[11.3]{Jech}, $|\D|^2=|\D|$ for all infinite $\D$ is already equivalent to the axiom of choice. 

Since hypostatic abstraction can reduce ternary to binary relations it does not reduce functions to functions, or we could decompose $2$-input functions into $1$-input ones. L\"owenheim's pairing construction \cite{Low} does not have this shortcoming, but it requires adjoining (or encoding) not just $n$-tuples, but also $2$-tuples, $3$-tuples, ..., $(n-1)$-tuples of the domain elements. We will describe the functional version of it used by Sierpi\'nsky \cite{Sier45}, who was apparently unaware of L\"owenheim's work, to prove FRT on infinite domains. 

After extending the domain, define recursively 
$$
i_2(x_1,x_2):=(x_1,x_2),\ \ \ \ i_k\big((x_1,\dots, x_{k-1}),x_k\big):=(x_1,\dots,x_k)
$$ 
for $3\leq k\leq n$, and define $\widehat{f}((x_1,\dots,x_n)):=f(x_1,\dots,x_n)$. Then extend $i_k$ and $\widehat{f}$ arbitrarily to other elements. It is easy to see by induction that 
$$
f(x_1,\dots,x_n)=\widehat{f}\Big(i_n\big(\dots i_4\,\big(\,i_3\,(\,i_2\,(x_1,x_2),x_3),x_4\big)\dots, x_n\big)\Big)
$$ 
reduces $f$ to a composition of $2$-input $i_k$ and $1$-input $\widehat{f}$ (which can be turned into a $2$-input by adding a dummy variable if one so wishes). Again, to avoid extending the domain one can encode the adjoined elements into the original ones. This is always possible on infinite domains, but requires the axiom of choice, as Sierpi\'nsky explicitly noted. We can summarize the constructions in this section in the following theorem that combines the reducibility clauses of RRT and FRT on infinite domains.
\begin{theorem}\label{InfRed} Let $n\geq3$. Any $n$-ary relation  on an infinite domain is a composition of $n$ binary relations. Any $n$-input function on an infinite domain is a composition of $n$ $2$-input functions. 
\end{theorem}
It is doubly ironic that the (ostensibly) harder case of relations was settled several decades earlier than the case of functions, and that the case of functions on finite domains, where one cannot produce simple explicit constructions, was settled before the one on infinite domains.

\section{Functions, relations, relatives}\label{FunRel}

In this section we will give an elementary construction that reduces relations on a domain to ternary relations using that functions on it can be reduced to $2$-input functions. As with the hypostatic abstraction and the pairing construction, we will need the axiom of choice for the general case. However, the construction is of most interest on finite domains, where it is unnecessary and the reductions can be performed constructively. That any function on a finite domain reduces to a composition of $2$-input functions follows from Post's results for many-valued logics discussed in the next section.

Considered as relations, functions have a distinguished argument designated as output. Our approach will be to analogize relations to functions, so we will extend the notion of output to general relations. Linguistic expressions that Aristotle, de Morgan and Peirce took as motivation, such as ``brother of \underline{\hspace{1em}}" or ``son of \underline{\hspace{1em}} by \underline{\hspace{1em}}", distinguish the output explicitly in the noun phrases. We will follow them in calling relations with a distinguished argument {\it relatives} \cite{Bur97}. In clone theory, relatives on finite domains are called ``hyperoperations" \cite{Ros96}, ``partial hyperoperations" \cite{Rom02} or ``multifunctions" \cite{Born08}. 
\begin{definition}\label{Relative}
A {\bf relative} is a relation on $\D$ with a distinguished argument called the {\bf output}. The remaining arguments are called {\bf inputs}. Unless otherwise stated, the output is designated to be the last listed argument, e.g. $x_{n}$ in $R(x_1,\dots,x_{n})$. For $\a\in\D^{n-1}$ we define its {\bf image under $R$} as the set $R(\a):=\{a\in\D\,|\,(\a,a)\in R\}$. We call a relative {\bf partial function} when $R(\a)$ is empty or a singleton for every $\a$, and a {\bf multifunction} when $R(\a)$ is non-empty for every $\a$. When $R(\a)$ is a singleton for every $\a$, we call $R$ the {\bf graph of a function} $\D^{n-1}\to\D$.
\end{definition}
\noindent Thus, general relatives can be characterized as `partial multifunctions'. It will be convenient for us to interpret the domain of a relative as itself a relative rather than as just a set.
\begin{definition}\label{RelDom}
For any relative $R(x_1,\dots,x_n)$, we define the domain relative 
$$
\textup{Dom}_R(x_1,\dots,x_{n-1}):=\exists\,t\big[R(x_1,\dots,x_{n-1},t)\big]
$$ 
as the projection on its inputs.
\end{definition}
Note that $\textup{Dom}_R$ is a relative of lower arity. The following lemma codifies that any relative is a multifunction on its domain. The proof is straightforward and we omit it.
\begin{lemma}\label{DomSplit} Any relative $R$ decomposes as
\begin{equation}\label{Fdom}
R(x_1,\dots,x_n)=F(x_1,\dots,x_n)\land\textup{Dom}_R(x_1,\dots,x_{n-1}),
\end{equation}
where $F$ is a multifunction. If $R$ is a partial function then $F$ can be chosen to be a function, i.e. $F(x_1,\dots,x_n)$ if and only if $x_n=f(x_1,\dots,x_{n-1})$.
\end{lemma}
Thus, any relation can be reduced to a multifunction and a relation of lower arity. In turn, we will reduce multifunctions to (ordinary) functions by using value selectors. 
\begin{definition}\label{Select}
Let $F$ be a multifunction on $\D$. A function $f:\D^{n-1}\to\D$ is called a {\bf value selector} for $F$ when for any $\a\in\D^{n-1}$ we have 
$(\a,f(\a))\in F$, i.e. $f(\a)\in F(\a)$.
\end{definition}
When $\D$ is infinite, existence of selectors depends, in general, on the axiom of choice. There are $\prod_{\a\in\D^{n-1}}|F(\a)|$ possible selectors, which can run up to $\big(\D^{n-1}\big)^{|\D|}=|\D|^{(n-1)|\D|}$ in cardinality. However, it will be important that at most $|\D|$ are needed to cover all tuples in $F$. This is exploited in the next lemma.
\begin{lemma}\label{MultSel} Let $\D$ be a well-orderable set. Then any $n$-ary multifunction $F$ on $\D$ decomposes as
\begin{equation}\label{MSel}
F(x_1,\dots,x_n)=\exists\,t\big[x_n=f(t,x_1,\dots,x_{n-1})\big],
\end{equation}
where $f:\D^{n}\to\D$ is a function.
\end{lemma}
\begin{proof}
Well-order $\D$ and let $f_1$ be the value selector of $F$ that picks the least element of $F(\a)$ at each $\a\in\D^{n-1}$. For every ordinal $\sigma$ we define $f_{\sigma}$ to pick the least element of $F(\a)$ not picked by any $f_{\tau}$ with $\tau<\sigma$, if any, and the least element of $F(\a)$ otherwise. By transfinite induction, at some ordinal $\xi$ with $|\xi|\leq|\D|$ there will be no unpicked elements left for any $\a\in\D^{n-1}$, because $|F(\a)|\leq|\D|$. We put $\xi$ into $1$-$1$ correspondence with a subset $T$ of $\D$ and define $\sigma(t)$ to be the image of $t$ under this correspondence when $t\in T$ and $1$ otherwise. Finally, we set $f(t,x_1,\dots,x_{n-1}):=f_{\sigma(t)}(x_1,\dots,x_{n-1})$. Since $f_{\sigma}$ is a selector for every $\sigma$ only tuples in $F$ will satisfy $x_n=f_{\sigma(t)}(x_1,\dots,x_{n-1})$ for any $t$, and all of them are accounted for because every $\sigma$ in $\xi$ is $\sigma(t)$ for some $t$. Thus, \eqref{MSel} holds.
\end{proof}
We will refer to $f(t,x_1,\dots,x_{n-1})$ as the {\bf indexing function} (of value selectors). When $\D$ is finite, it is obtained constructively, and if we identify $\D$ with the truth value set of a many-valued logic then we even have a natural well-order on $\D$ that defines it uniquely. It is known that any $n$-input function with $n\geq3$ on any domain decomposes into $2$-input functions (see Section \ref{PostLog} for finite domains). Relations of the form $u=f(t,x)$ can then be reduced to ternary relations by recursively unfolding nested compositions. 
Combining this procedure with Lemmas \ref{DomSplit} and \ref{MultSel}, we get our main result.
\begin{theorem}[{\bf Relative reduction}]\label{FunRed} Any $n$-ary relation $R$ on any domain $\D$ decomposes as
\begin{equation}\label{Fred}
R(x_1,\dots,x_n)=\exists\,t\big[x_n=f(t,x_1,\dots,x_{n-1})\big]\land\textup{Dom}_R(x_1,\dots,x_{n-1}),
\end{equation}
where $f:\D^{n}\to\D$ is a function. Moreover, given a functionally complete set, every relation can be decomposed into graphs of functions from this set. In particular, every $n$-ary relation on $\D$ with $n\geq4$ is reducible to ternary relations.  
\end{theorem}
\begin{proof} Formula \eqref{Fred} is a direct consequence of Lemmas \ref{DomSplit} and \ref{MultSel}. To decompose the first conjunct in \eqref{Fred}, set $x_0:=t$ and suppose
$$
f(t,x_1,\dots,x_{n-1})=g\big(h_1(x_{\L_1}),\dots,
h_m(x_{\L_m})\big),
$$
where $\L_i\subseteq\{0,1,\dots,n-1\}$ and $x_{\L}$ is the list of $x_i$ with $i\in\L$. Then
\begin{multline}\label{Fun2Rel}
\exists\,t\big[x_n=f(t,x_1,\dots,x_{n-1})\big]\\ 
= \exists\,t\exists\,t_1\dots\exists\,t_m
\big[x_n=g(t_1,\dots,t_m)\land t_1=h_1(x_{\L_1})\land\dots\land t_m=h_m(x_{\L_m})\big].
\end{multline}
By iterating the procedure, if necessary, we can decompose the first conjunct into graphs of functions from a given complete set. The same process is then applied to $\textup{Dom}_R$. Since the arity is reduced by $1$ at each step the process terminates in finitely many steps. Since there exist functionally complete sets of $2$-input functions, and their graphs are ternary, the last claim follows.
\end{proof}
\noindent Theorem \ref{FunRed} not only extends Theorem \ref{InfRed} to all domains, but its proof also reduces RRT to FRT, as Peirce's remarks suggested. However, the result is weaker. The hypostatic abstraction decomposed all relations into {\it binary} relations and bounded their number by arity. We will mend the first flaw in Section \ref{RedBin}, but there can be no uniform arity bound on finite domains \cite{Kosh23}. 

Recall that a function is called  Sheffer when it alone forms a functionally complete system \cite[11.2]{Lau},\cite{Ros76}. Peirce found one such function, NOR, a.k.a. Peirce arrow, on Boolean domains in an unpublished 1902 manuscript. Sheffer published the result for NAND, a.k.a. Sheffer stroke, in 1913 \cite{AnelTT}. On general finite domains, the first Sheffer function was found by Webb in 1935, and it is a $2$-input one \cite{Ros76,Zus16}. Their graphs give us ternary Sheffer relations on the corresponding domains, a result that appears to be new.
\begin{corollary}[{\bf Sheffer relations}]\label{ShefRel} On any finite domain, there exist ternary relations that alone compose all relations.
\end{corollary}
Let us illustrate the construction of Theorem \ref{FunRed} with a couple of examples.
\begin{example} Consider the $n$-identity relation $I_n$, whose tuples are all and only $n$-tuples of identical elements. Taking $x_n$ as the output, we see that $I_n$ is a partial function with $\textup{Dom}_{I_n}=I_{n-1}$. On this domain, $I_n$ can be identified with a function that extends to all of $\D$ as e.g. $f(x_1,\dots,x_{n-1}):=x_1$. This means, according to \eqref{Fdom}, that
\begin{equation*}
I_n(x_1,\dots,x_n)=\big[x_n=x_1\big]\land I_{n-1}(x_1,\dots,x_{n-1}).
\end{equation*}
Iterating $n-1$ times we obtain a familiar decomposition of $I_n$ into diagonals, namely
\begin{equation*}
I_n(x_1,\dots,x_n)=\big[x_n=x_1\big]\land\dots\land\big[x_2=x_1\big].
\end{equation*}
\end{example}

\begin{example}\label{NegIred2} The $n$-identity is small even on finite domains, and hence can be decomposed into binaries by hypostatic abstraction even there. In contrast, $|\neg I_n|=n^{|\D|}-n>n$ for $n\geq3$ and $|\D|\geq2$, i.e. for $n\geq3$ the non-$n$-identity is large already on Boolean domains. We will decompose it there first, where we can take advantage of the familiar Boolean algebra, and postpone the case of general finite domains until the next section, after Post's many-valued logics are introduced for the task. 

Note that $\textup{Dom}_{\neg I_n}=U_{n-1}$ is the universal relation (because any $(n-1)$-tuple can be complemented to an $n$-tuple with not all entries equal), and we can drop it from the decomposition in \eqref{Fred}. It remains to pick value selectors for the multifunction $\neg I_n$ and decompose the resulting indexing function into $2$-input functions.

In the Boolean case we can identify $\D$ with $\{0,1\}$ and pick  
$$
f_0(x):=\begin{cases}0,\,x\neq(0,\dots,0)\\
1,\,x=(0,\dots,0)\end{cases}\ \ \ \ \ \ 
f_1(x):=\begin{cases}1,\,x\neq(1,\dots,1)\\
0,\,x=(1,\dots,1)\end{cases}
$$
as the value selectors. Then, using the standard notation for Boolean operations,
\begin{gather}
f_0(x)=\nx_1\land\dots\land\nx_{n-1},\ \ \ \ \ \ 
f_1(x)=\nx_1\lor\dots\lor\nx_{n-1};\notag\\\label{negTf}
\text{so }f(t,x)=t\land\overline{x_1\land\dots\land x_{n-1}}\,\lor\,\nt\land\nx_1\land\dots\land\nx_{n-1} 
\end{gather}
satisfies $f(i,x)=f_i(x)$ for $i=0,1$. Since \eqref{negTf} expresses $f(t,x)$ via $2$-input and $1$-input functions the template  \eqref{Fun2Rel} will decompose $\neg I_n$ into ternary and binary relations of the form $z=x\land y$, $z=x\lor y$ and $y=\nx$.
\end{example}

\section{Reduction on Post's logics}\label{PostLog}

This section is expository and can be skipped without loss of continuity. Its purpose is to demonstrate that the reduction procedure of Theorem \ref{FunRed} is effective on all finite domains, and quite practical considering the wealth of functional decomposition methods. We recall the analogs of Boolean operations in Post's many-valued logics and of the canonical disjunctive normal form that reduces any function on them to $2$-input and $1$-input functions. The construction is well-known, see e.g. \cite[1.4]{Lau}, \cite{Ros76}, \cite[2.8]{Urq01}, but is rarely spelled out explicitly. As an illustration, we generalize to finite domains the Boolean reduction of $\neg I_n$.

From the modern perspective, a simple way to introduce (most of) Post's $(k+1)$-valued logic is by restricting the standard fuzzy logic with truth values in $[0,1]$ to the subset $\ds{\left\{0,\frac1k,\dots,\frac{k-1}k,1\right\}}$. However, it is more convenient to work with integers, so, as is customary in the clone theory, we multiply the truth values by $k$. Then their set becomes $\ds{E_{k+1}:=\left\{0,1,\dots,k-1,k\right\}}$, and the standard connectives are defined by
$$
a\land b:=\min(a,b);\ \ \ \ a\lor b:=\max(a,b);\ \ \ \ \overline{a}:=k-a.
$$
For $k=1$ we recover the Boolean algebra $E_2$. Alas, for $k>1$ the system $\left\{\land,\lor,\overline{\phantom{a}}\right\}$ is not functionally complete. Instead of $\,\overline{\phantom{a}}\,$, called the {\it diametric negation}, Post introduced the {\it cyclic shift} 
$$
a':=a+1\pmod{k+1},
$$
sometimes also called cyclic negation since $a'=\overline{a}$ in $E_2$. The system $\left\{\land,\lor,\,'\right\}$ turns out to be functionally complete.

To see this, first note that for any $x$
$$
0=x\land x'\land\dots\land x^{(k)},\, 1=0',2=0''\dots,\,k=0^{(k)},
$$
so we can generate all the constant functions. Next, consider the shifts $x^{(j)}$ with $j\neq i$. One of them is $k$ unless $k=x^{(i)}$, i.e. $x=k-i$. Therefore,
$$
\bigvee_{j\neq i} x^{(j)}=\begin{cases}k-1,\,x=k-i\\
k,\,x\neq k-i\end{cases}\!\!\!\!\!\!,
$$
and we can generate the {\it literals} (a.k.a. value isolators) \cite{Ros76}:
$$
x^i:=\begin{cases}k,\,x=i\\
0,\,x\neq i\end{cases}\!\!\!\!\!
=\left(\bigvee_{j\neq k-i} x^{(j)}\right)'.
$$
In the Boolean algebra, $x^0=\nx$ and $x^1=x$. From the literals we generate the analogs of Boolean monomials ($x$ and $\a$ below are now $n$-tuples of variables and truth values, respectively):
\begin{equation}\label{liter}
x^\a:=x_1^{\a_1}\land\dots\land x_n^{\a_n}\,
=\begin{cases}k,\,x=\a\\
0,\,x\neq \a\end{cases}.
\end{equation}
Since $k$ is maximal (the ``true") $i\land x^\a$ returns the given $i$ on $\a$ and $0$ otherwise. With the above,  any function $f:E_{k+1}\to E_{k+1}$ decomposes into the {\it canonical disjunctive normal form} (CDNF) \cite[1.4]{Lau}, \cite{Ros76}, which can be read off from its truth table as in the Boolean case (taking the empty disjunction to return $0$):
\begin{equation}\label{CDNF}
f(x)=\bigvee_{i\in E_{k+1}\-\{0\}} 
\left(\bigvee_{f(\a)=i}i\land x^\a\right).
\end{equation}
For the diametric negation, we get, for example, $\ds{\nx=\!\!\!\!\!\!\bigvee_{i\in E_{k+1}\-\{0\}}\!\!\!\!\!\!\!\!i\land x^{k-i}}$. In the Boolean case, CDNF reduces to the familiar $\ds{f(x)=\bigvee_{f(\a)=1}x^\a}$.

Note that CDNF directly uses, in addition to the $2$-input $\land,\lor$, only $0$-input constants and $1$-input literals, so we can replace the cyclic shift with them, and get another complete system of functions with at most $2$ inputs. Webb's $2$-input Sheffer function, a multi-valued generalization of Sheffer stroke, is given by $(a\lor b)'$.

Identifying $\D$ with $E_{k+1}$ for $k:=|\D|-1$, we get $2$-input decompositions on any finite domains. Let us illustrate the above constructions by generalizing Example \ref{NegIred} to the case of arbitrary $k\geq1$.

\begin{example}\label{NegIred} As in the Boolean case, $\neg I_n$ is a multifunction, so we only need to find its value selectors and decompose their indexing function $f(t,x_1,\dots,x_{n-1})$ into CDNF. Let $\vec{i}:=(i,\dots,i)$. With $\D=E_{k+1}$, the value selectors can be
$$
f_i(x):=\begin{cases}i,\,x\neq\vec{i}\\
i',\,x=\vec{i}\end{cases}\!\!\!.
$$
Indeed, $f_i$ select from $\neg I_n$ because $i\neq i'$ for any $i$, and they exhaust $E_{k+1}$ because every $(n-1)$-tuple other than $\vec{i}$ is evaluated to $i$ by $f_i$, while $\vec{i}$ is evaluated to any $j\neq i$ by $f_j$. These selectors can be uniformized as $f_i(x)=i\land\overline{x^{\vec{i}}}\lor i'\land x^{\vec{i}}$. As a result, we get an expression for the indexing function analogous to the Boolean one \eqref{negTf}, with the cyclic shift and diametric negation replacing the Boolean negation in different places:
$$
f(t,x)=t\land \overline{x_1^{t}\land\dots\land x_{n-1}^{t}}\,\lor\,t'\land x_1^{t}\land\dots\land x_{n-1}^{t}\,.
$$
To decompose $\neg I_n$ into ternaries and binaries  on any $E_{k+1}$, it remains to repeatedly apply the template \eqref{Fun2Rel}.
\end{example}

\section{Ternary from binary relations}\label{RedBin}

A natural question, left unanswered by Theorem \ref{FunRed}, is whether ternary relations reduce to binary relations on finite domains, as they do on infinite domains by hypostatic abstraction. This is the question we will take up in this section. Converting a function into a relation requires adding an extra argument to it, so binary reduction along the same lines would require us to compose all functions from $1$-input ones, which is, obviously, impossible. If the binary reduction is possible it must exploit a different idea.

Our idea is to reduce cofinite relations of small {\it arity} (up to the cardinality of the domain) to binary relations. In a sense, it is dual to hypostatic abstraction, which reduces relations of small {\it cardinality}. Unlike the BKKR construction \cite{BKKR}, or the non-constructive proof via the Galois connection \cite{HerPos04}, \cite[1.1.22]{PosKal}, our construction, together with Theorem \ref{FunRed}, can be said to explain `why' RRT is true when $|\D|\geq3$ and suggest `why' it breaks down on Boolean domains. Of course, the failure of a construction does not establish the negative claim in and of itself, but it is known that RRT fails on Boolean domains and irreducible ternary relations do exist \cite{HerPos04}. For example, $\neg I_3$ is one. For completeness, we will give an elementary proof in the next section.

Treating relations on $\D$ as predicates, i.e. functions $\D^n\to E_2$, we can apply Boolean operations to them. This allows to easily build many relations from fairly simple primitives. However, Boolean formulas do not decompose them, in our sense, into the primitives when they involve non-positive operations $\neg$ or $\lor$. We will write $I_2(x,y)$ and $\neg I_2(x,y)$ in the more traditional form $x=y$ and $x\neq y$, respectively, and bracket them when confusion with the equality of formulas may result. 

The predicates $x=a$ coincide with the literals $x^a$ from \eqref{liter} if we identify $\D$ with $E_{k+1}$ and $E_2$ with $\{0,k\}\subseteq E_{k+1}$, so we will also call them literals. Conjoining the literals, we express the singleton relations for $\a\in\D^n$ as monomials 
\begin{equation}\label{Rmon}
[x=\a]:=[(x_1=\a_1)\land\dots\land(x_n=\a_n)]=\bigwedge_{j=1}^n(x_j=\a_j),
\end{equation}
where now $x:=(x_1,\dots,x_n)$. Any finite cardinality relation on $\D$ can be built from them using disjunctions, namely
\begin{equation}\label{RCDNF}
R(x)=\bigvee_{\a\in R}(x=\a)\,.
\end{equation}
This is just the predicate version of Boolean CDNF.

In decompositions, aside from conjunctions, we can only use existential quantification. One can emulate some disjunctions with it, but only when the number of disjuncts is at most $|\D|$ and at the expense of increasing their arity by $1$. The idea is to convert the index into a (bound) variable, as in $\bigvee_jR_j(x)=\exists t\,R(t,x)$, The literals are unary, so converting them into binary relations is acceptable, but the number of disjuncts in the outer disjunction of \eqref{RCDNF} is $|R|$. This can go up to $|\D|^n$ for $n$-ary $R$, so we cannot convert CDNF into a ppf in general. When it does work, i.e. for small cardinality relations, it reproduces the hypostatic abstraction \eqref{HypAbs}. This is a dead end.

It could have been anticipated if we noticed that both conjunction (intersection) and existential quantification (projection) do not increase the cardinality of relations. If we are to compose all relations we must start from relations of largest cardinality, not the singletons. Aside from the universal relation $U_n$, those are the cosingletons expressed by comonomials
\begin{equation}\label{Rcomon}
[x\neq\a]:=[(x_1\neq\a_1)\lor\dots\lor(x_n\neq\a_n)]=\bigvee_{j=1}^n(x_j\neq\a_j)\,.
\end{equation}
Negating the CDNF for $\neg R$ we obtain the predicate version of CCNF
\begin{equation}\label{RCCNF}
R(x)=\bigwedge_{\a\not\in R}(x\neq\a)\,.
\end{equation}
for any cofinite $R$. On finite domains, all relations are both finite and cofinite, so CCNF decomposes them all into the comonomials.

Alas, the arity of $x\neq\a$ must match the arity of $R$, and their Boolean reduction \eqref{Rcomon} to unary coliterals is not a ppf. However, the disjunction there is not as bad as in CDNF. The number of disjuncts is only $n$, so for $n\leq|\D|$ the essential condition of converting disjunction into existential quantification is met. There is also a technical issue that disjuncts coming from turning $\exists t\,R(t,x)$ into $\bigvee_jR_j(x)$ must have the same list of variables, which is not the case in \eqref{Rcomon}. It can be handled as follows. Note that 
$$
[x_j\neq\a_j]=\bigwedge_{i=1}^n\left[(j\neq i)\lor(x_i\neq\a_i)\right]
$$
since the only non-trivial conjunct has $j=i$, and the right hand side now lists all $n$ variables. When $n\leq|\D|$, fix an injection $\pfi:\{1,\dots,n\}\to\D$, and define the unary relation $\textup{Ran}_\pfi(t)$ that holds on the range of $\pfi$. Now replace $j$ by an indexing variable $t$, and note that $j\neq i$ is equivalent to $\pfi(j)\neq\pfi(i)$. For $i=1,\dots,n$ and $a\in\D$ define the binary relations
$$
B_{i,a}(t,x):=\left[\big(t\neq\pfi(i)\big)\lor\big(x\neq a\big)\right]\land\textup{Ran}_\pfi(t).
$$
When $\D$ is identified with $E_{k+1}$ one can take $\pfi(i):=i-1$. The centerpiece of our construction is a reduction of low arity comonomials $x\neq\a$ to these binary relations by `existentialization' of their defining disjunction \eqref{Rcomon}. It is similar to the procedure used in the proof of Lemma \ref{MultSel} to `existentialize' the disjunction of value selectors.
\begin{theorem}\label{CoBin} For any $3\leq n\leq|\D|$ any cofinite $n$-ary relation on a domain $\D$ reduces to binary relations as follows:
\begin{equation}\label{CofinRed}
R(x)=\bigwedge_{\a\not\in R}\exists t\!\left[\,\bigwedge_{i=1}^nB_{i,\a_i}(t,x_i)\right]\!.
\end{equation}
\end{theorem}
\begin{proof} The existentially quantified conjuncts in \eqref{CofinRed} simply reduce $x\neq\a$ to binary relations $B_{i,a}(t,x)$. Indeed, due to the conjunction with $\textup{Ran}_\pfi(t)$, we can replace $t$ by $\pfi(j)$ and $\exists t$ by disjunction over $j$ without changing the corresponding relation. The disjuncts then simplify to
\vspace*{-0.9em}
\begin{multline*}
\bigwedge_{i=1}^nB_{i,\a_i}\big(\pfi(j),x_i\big)
=\bigwedge_{i=1}^n\left[\big(\pfi(j)\neq\pfi(i)\big)\lor\big(x_i\neq\a_i\big)\right]\\
=\bigwedge_{i=1}^n\left[\big(j\neq i\big)\lor\big(x_i\neq\a_i\big)\right]
=[x_j\neq\a_j]. 
\end{multline*}
By \eqref{Rcomon}, their disjunction further simplifies to just $x\neq\a$. It remains to apply CCNF \eqref{RCCNF} to the outer conjunction over $\a$.
\end{proof}
Combining this result (even just for $n=3$) with Theorem \ref{FunRed} we can now reduce any cofinite relation to binary ones as long as $|\D|\geq3$. For $|\D|=2$ the above construction obviously breaks down because the injection $\pfi$ does not exist. There are not enough elements in Boolean domains to emulate triple and higher disjunctions by existential quantification. Although this does not prove irreducibility of some ternary relations, it strongly suggests it. 

While the BKKR reduction \cite{BKKR} is less transparent than ours and works only on finite domains, it does prove a stronger result: relations of higher arity are reduced not just to {\it some} binary relations, but to two specific ones. Let us briefly sketch it for comparison. Identifying $\D$ with $E_{k+1}$, we have the binary relation $\leq$ on it since $E_{k+1}$ is well-ordered. From $\leq$ and $\neq$, we can compose $<$ as $[x<y]=[(x\leq y)\land(x\neq y)]$. The constants and literals are then composed using $<$. Negations of the literals $x=a$ are composed from them and $\neq$ as $\exists t\,[(x\neq t)\land (t=a)]$.
The ternary relation 
$$
Q_3(x,y,z):=[(x=y)\to(x=z)]=[(x\neq y)\lor(x=z)]
$$ 
plays a key role in the BKKR construction \cite{BKKR}. Its defining formula is not a ppf, but BKKR do explicitly reduce $Q_3$ to $\leq$ and $\neq$ for $|\D|\geq3$. A rather long and opaque ppf then reduces cosingletons $x\neq\a$ to $Q_3$ and some binary relations, and an application of CCNF \eqref{RCCNF} completes the reduction, as in our case. When reducing cosingletons, BKKR have to separately consider different patterns of identical entries in them. Much of that complexity is removed in our construction by outsourcing reduction of high arity cosingletons to functional decompositions in the relative reduction.

\section{Boolean wrinkle}\label{BoolIrr}

We will now confirm that the exclusion of Boolean domains from Theorem \ref{CoBin} is not an artifact of the construction. The fact that some ternary relations are irreducible on Boolean domains is well-known and follows from the structure of the coclone lattice, which is dual to the Post lattice of Boolean clones. The first explicit proof seems to be due to Schaefer \cite{Sch78}, who did not use clone theory, but it is  somewhat convoluted. In the spirit of the paper, we will give a simple elementary proof.

Boolean relations, when identified with their predicates, coincide with Boolean functions as sets, both are maps $E_2^n\to E_2$. Therefore, relations can be represented by standard Boolean expressions, like $x\oplus y$ for $x\neq y$, or $x\to y$ for $x\leq y$. Relational CDNF \eqref{RCDNF} and CCNF \eqref{RCCNF} then turn into the familiar functional versions, but the composition is different. Functional composition naturally aligns with substitutions into standard connectives, while relational pp-composition is less intuitive. We can split it into three operations: taking conjunctions, identifying variables in predicates and existentially quantifying. The former two translate into the corresponding operations on representing functions, while the latter becomes
\begin{equation}\label{E2Fun}
\exists\,t\big[R(x_1,\dots,x_{n-1},t)\big]=R(x_1,\dots,x_{n-1},0)\lor R(x_1,\dots,x_{n-1},1).
\end{equation}

Of particular interest to us are relations called {\it bijunctive} \cite{CreVol08, Sch78}, which are defined, in terms of their truth value functions, as those admitting conjunctive normal forms with at most two literals (2CNF). Recall that conjuncts of CNF are, in general, finite disjunctions of literals called {\it clauses}. Clearly, all unary and binary relations are bijunctive, as their CCNF is a 2CNF. Moreover, bijunctivity is preserved by pp-composition. 
\begin{lemma}\label{BijComp}
A pp-composition of bijunctive relations is bijunctive.
\end{lemma}
\begin{proof}
Conjunction of 2CNF is trivially a 2CNF, and identification of variables either leaves 2-clauses as 2-clauses, collapses them into 1-clauses (when the literals coincide), or into constant $1$ (when the literals are complementary). So the result is still a 2CNF. For the quantification, note that substituting $0$ or $1$ into 2-clauses turns them into either $1$-clauses or constants. Then $1$-s can be removed from the resulting conjunction, and $0$-s would turn it into $0$. Unless it is the latter, the disjunction in \eqref{E2Fun} becomes a disjunction of two monomials, and distributing $\land$ over $\lor$ turns it into a 2CNF. 
\end{proof}
Thus, to show that a Boolean relation does not decompose into binary ones it suffices to show that it is not bijunctive. That we can do by elementary means already for ternary relations.
\begin{theorem}
$\neg I_3$ is non-bijunctive, and hence irreducible to binary relations on Boolean domains.
\end{theorem}
\begin{proof}
By contradiction, suppose $\neg I_3$ has a 2CNF. Then $I_3$ has a 2DNF, a disjunction of monomials with at most two literals, obtained by negating the 2CNF and applying de Morgan laws. Unary monomials, like $x$ or $\overline{x}$, would make the entire disjunction equal $1$ for $x=1$ or $x=0$, respectively, regardless of $y$ and $z$ values. So no such monomials are there. There are no binary monomials with mismatched literals either. Indeed, say $x\land\overline{y}$ or $\overline{x}\land y$ would make the disjunction equal $1$ for unequal values of $x$ and $y$. Finally, binary monomials with matched literals, like $x\land y$ or $\overline{x}\land\overline{y}$, would make the disjunction equal $1$ for $x=y=1$ or $x=y=0$, respectively, regardless of $z$ values. But if no monomials are there, the supposed 2DNF of $I_3$ must be a constant, contradiction.
\end{proof}
We now have at least one ternary relation that is not reducible to binary ones on Boolean domains, and we can present several more. Theorem \ref{FunRed} decomposes all relations into ternary relations of the form $z=g(x,y)$, the graphs of $2$-input functions used in functional decompositions. Consider any $2$-input Boolean function $g$ that alone, or together with $1$-input negation, forms a functionally complete system. If the ternary relation $z=g(x,y)$ were bijunctive then, by Theorem \ref{FunRed}, every relation would have been bijunctive. Thus, the graphs of $\land$, $\lor$, $\to$, $\leftarrow$, $|$ (Sheffer stroke), $\downarrow$ (Peirce arrow) are all non-bijunctive.

Leaving aside the elementary approach, one can prove much stronger results. Bijunctive relations form a maximal coclone, meaning that adding any non-bijunctive relation to them gives a relationally complete set \cite{BRSV05,CreVol08}. Moreover, this coclone is pp-compositionally generated by binary relations. In fact, one can take $\leq$ and $\neq$, those two relations that generate the coclone of {\it all} relations when $|\D|\geq3$ by the BKKR construction, as the bijunctive generators for $E_2$ \cite{BRSV05}. This means that $\leq$,  $\neq$, $\neg I_3$ generate all Boolean relations, and one can replace $\neg I_3$ by any non-bijunctive Boolean relation of any arity.

There is also an alternative characterization that allows to test relations for bijunctivity based on their truth value set. It was proved by Schaefer, but was probably known earlier. To describe it, define the $3$-input majority function as follows\begin{equation}\label{3Major}
\mu_3(x,y,z):=(x\land y)\lor(y\land z)\lor(x\land z).
\end{equation}
It returns $1$ when a majority of its inputs is $1$, and $0$ when a majority is $0$. It turns out that bijunctive relations are all and only those preserved by $\mu_3$. That is, if we extend $\mu_3$ entrywise to a tuple-valued function on tuples then $\mu_3(\a,\beta,\gamma)$ must be in the relation whenever $\a,\beta,\gamma$ are. 

By the Pol-Inv Galois connection, maximality of the bijunctive coclone is equivalent to minimality of the clone generated by $\mu_3$. On non-Boolean domains, there are many $3$-input functions that return the majority value when there is one. There are $729$ of them already on $3$-element domains \cite{Ros86}. From this point of view, the reason all relations can be reduced to binary relations on non-Boolean domains is that none of those majority functions preserves all binary relations.

\section{Bonds and teridentity}\label{TerId}

The original version of the reduction thesis for relations goes back to Peirce \cite{Bur97,Kosh22}. However, he operated with a more restrictive notion of composition that iterates relative products and excludes many ppfs, making reduction to binary relations impossible even on non-Boolean domains. After L\"owenheim's result on binary reducibility \cite{Low}, Peirce's thesis became controversial \cite{Kosh22}, and was not treated mathematically until the formalizations of Herzberger \cite{Herz} and Burch \cite{Bur91} in 1980-s. And its connection to universal algebra and pp-decompositions was only pointed out in 2004 by Hereth and P\"oschel \cite{HerPos04}. 

In this section, we will show, following Burch \cite{Bur91,Bur97}, that, somewhat surprisingly, for arities $n\geq4$ Peirce's restricted composition leads to the same notion of reducibility as pp-composition. However, some ternary relations, are irreducible on Peircean definition. One such relation, the ternary identity $I_3$, {\bf teridentity} for short, plays a key role in converting pp-decompositions into Peircean ones, and is the only ternary relation needed in Peircean reductions.

The relative product is a binary operation on predicates where one inserts the same variable into a single argument of each and existentially quantifies over it. As we saw in the Introduction, this is the most direct generalization of the composition of two functions. One can also insert the same variable into {\it two} different arguments of the same predicate and existentially quantify over it. These are the only ways to identify variables, so only free conjunctions are allowed, no joins. The triple junction $\exists t\left[P(t,x)\land Q(t,y)\land R(t,z)\right]$ and higher cannot be generated by Peirce's means either, because the identified variables are quantified over and no new variable can be identified with them afterwards. But those junctions are exactly the ones featured in hypostatic abstraction \eqref{HypAbs}.

To summarize, in predicate formulas produced by iterating Peirce's operations no variable can occur in more than two positions, and if it does it has to be a bound variable. We will adopt Herzberger's term ``bond" for this restricted class of compositions, although our bond is slightly more permissive than his, along the lines of \cite{HerPos04}. The name is inspired by the analogy with chemical bonding promoted by Peirce himself. For a detailed study of bonding and its relation to graph theory, see \cite{Kosh23}. 
\begin{definition}[{\bf Bond}]\label{BondDef}
A {\bf bonding formula (bf)} is a primitive positive formula where no free variables repeat, and bound variables repeat at most twice. We say that a relation is a {\bf bond of relations}, or that it is {\bf bonded} from them, when its predicate can be expressed by a bf with their predicates being the only predicate symbols in it. A bond is a {\bf bond reduction} when the bonded relations have strictly lower arity than their bond. A relation is {\bf bond reducible} when it admits a bond reduction.
\end{definition}
Peirce's first observation was that multiple variable identifications bond reduce to pairwise ones by using $n$-ary identity predicates $I_n$, e.g.
\begin{multline}\label{multiidn}
R_1(t,x_{\L_1})\land\dots\land R_n(t,x_{\L_n})=\\
\exists t_1\dots\exists t_n\left[I_{n+1}(t,t_1,\dots,t_n)\land R_1(t_1,x_{\L_1})\land\dots\land R_n(t_n,x_{\L_n})\right].
\end{multline}
And his second observation was that $I_n$ for $n\geq4$ bond reduce to teridentities $I_3$: 
\begin{multline}\label{InviaI3}
I_{n}(x_1,\dots,x_n)=\\
\exists t_1\dots\exists t_{n-3}\left[I_3(x_1,x_2,t_1)\land I_3(t_1,x_3,t_2)\land\dots\land I_3(t_{n-3},x_{n-1},x_n)\right].
\end{multline}
Applying the above identities, one can convert a pp-composition into a bond of the original relations and teridentities. As in \cite{Kosh22,Kosh23}, we call this conversion {\bf bond explication}. Since the explication adds only ternary relations we have the following theorem.
\begin{theorem}\label{Comp2Bond} 
For $n\geq4$, an $n$-ary relation is pp-reducible if and only if it is bond reducible. A ternary relation is pp-reducible if and only if it is a bond of binary relations and teridentities.
\end{theorem}
\begin{proof} The only claim not covered by the bond explication is that ternary  bonds of binary relations and teridentities are reducible. Given such a bond, assign a new variable to each teridentity and replace by it all occurrences of the original variables from the teridentity. Then remove the teridentity and the quantifiers over its variables. By \eqref{multiidn}, this produces an equivalent expression, and, since all teridentities are removed, it is a composition of binary relations only.
\end{proof}
\noindent In other words, bond reducibility to ternary relations is equivalent to pp-reducibility to binary relations. This justifies our talking of Peirce's reduction thesis despite using a different notion of composition. Ironically, L\"owenheim's result on binary reducibility \cite{Low} once made Peirce's reduction thesis controversial due to the confusion between these two closely related notions. However, while bond and compositional reducibilities are (almost) the same, bond reductions are much more special than general relational reductions, and typically involve many more ternary relations (teridentities). 

As for the ``almost", the two reducibilities do diverge on ternary relations. Ideologically, the proof goes back to Peirce himself, and variants of it are given in \cite{Bur91,HerPos04,Herz,Kosh22}.
\begin{theorem}\label{I3Irr} 
$I_3$ is bond irreducible on any domain with at least two elements.
\end{theorem}
\begin{proof} By contradiction, suppose we have a bf expressing $I_3(x,y,z)$ with only binary predicates in it. Since free variables are not shared there is a single predicate with $x$ in it. Let $t_1$ be the other variable in it. If it is shared, it is with a single other predicate and we let $t_2$ be the other variable in that one. And so on. This chain of variables stops at a variable not shared, either free or bound. None of the shared variables can be free, so our chain has at most two free variables (the first and the last). Moreover, variables from the chain can only appear in the predicates encountered in its construction, or they would be shared by more than two predicates.

Therefore, after prenexing quantifiers if necessary, we can split off the conjunction of the encountered predicates and represent our bf as a free conjunction of the form $P(x)\land Q(y,z)$, or $P(x,y)\land Q(z)$, or $P(x,z)\land Q(y)$. However, none of these can represent $I_3$. Indeed, suppose $I_3(x,y,z)=P(x,y)\land Q(z)$ and let $a\neq b\in\D$. Then $P(a,a)\land Q(a)$ and $P(b,b)\land Q(b)$ are both true, so $P(a,a)$ and $Q(b)$ are true, and hence $P(a,a)\land Q(b)$ must be true as well. But $I_3(a,a,b)$ is false, and analogous contradiction results in the other two cases.
\end{proof}
\noindent The proof implicitly employs graph-theoretic concepts, which are made explicit and developed in \cite{HerPos04,Kosh23}. It also shows that any bond reducible ternary relation must be a free conjunction (Cartesian product) of unary and/or binary relations. Peirce saw it as a vindication of a special role of ternary relations, and teridentity in particular, in actually {\it relating} their arguments, which Cartesian products fail to do \cite{Brun91,Kosh22}.

\section{Conclusions and open problems}\label{Conc}

We gave an elementary construction  of pp-reductions of higher arity relations to ternary and binary relations on finite domains. A key part of the proof reduces decomposition of relations to decomposition of functions, which is of general interest.

Unlike functional completeness, its relational analog is little studied. In particular, \cite{Ros76}, there are no general relational analogs of  completeness criteria, nor of full characterizations of maximal coclones on general finite domains. As a result, some basic questions long settled for functions remain open for relations. For example, our construction implies existence of Sheffer relations on any finite domain, but we do not know if there are any {\it binary} Sheffer relations on non-Boolean domains. The smallest known pp-complete system for them was constructed by BKKR and contains two binary relations, $\leq$ and $\neq$. There are some intriguing parallels between the role of unary relations in the CSP tractability conditions \cite{BJK} and of $1$-input functions in Slupecki's criterion of functional completeness \cite{Ros77} that suggest a possible direction of research.

There are even fewer completeness results for subsets of relations closed under the pp-composition (coclones), and it might be useful to relativize our construction to them. Recall that a set of relations that pp-compose all other relations in a coclone is called its {\it base} \cite{BRSV05}. Our Theorem \ref{FunRed} can be restated as saying that graphs of functionally complete systems are bases in the coclone of all relations. New bases of small arity for Boolean coclones were recently constructed in a number of works, including \cite{BRSV05,Lag17,Mar98}, but the non-Boolean case remains almost untouched. While not all coclones admit bases of consisting of function graphs even in the Boolean case, many do, and even those that do not have bases where graphs are supplemented by simple non-functional relations (like $\leq$). It is also of interest which coclones admit graph bases of small arity. The minimal required arity has been studied in the clone theory literature under the name of relation degree \cite[ch.\,10]{Lau}.

Finally, in a stark contrast to logic functions \cite{Ros76}, Shannon complexity of pp-decompositions into relational ``elements" has not been studied almost at all. Even qualitative analysis reveals a striking discontinuity between decompositions on finite and infinite domains that deserves further study \cite{Kosh23}. On infinite domains, one can reduce any $n$-ary relation to a projected conjunction of just $n$ binary relations. On finite domains, the structure of the decomposition is much more complex even for small arities, and the minimal number of binary relations in it grows with the size of the domain. The minimal number of ternary relations in bond reductions on finite domains has some information theory implications and is studied in \cite{Kosh23} in connection with Peirce's work.

\end{document}